\documentclass[12pt, reqno]{amsart}
\usepackage[utf8]{inputenc}
\usepackage[english]{babel}
\usepackage{multicol}
\usepackage{listings}
\usepackage{amsmath,blkarray,booktabs,bigstrut}
\usepackage{xcolor}
\usepackage{nicematrix}

\definecolor{backcolour}{rgb}{0.95,0.95,0.95}
\lstdefinestyle{mystyle}{
    backgroundcolor=\color{backcolour},
    stringstyle=\color{codepurple},
    basicstyle=\normalfont\ttfamily\scriptsize,  
    frame=ltb,
    framerule=0pt,
    breaklines=true,                 
    captionpos=b,  
}
\usepackage{booktabs} 

\usepackage{graphicx}
\usepackage{caption}
\usepackage{float}
\usepackage{listings}
\usepackage{verbatim}
\usepackage{todonotes}
\usepackage{subcaption}
\usepackage{enumerate} 
\usepackage{makecell}
    
\usepackage{indentfirst}

\usepackage[]{geometry}

\makeatletter
\newcommand*{\toccontents}{\@starttoc{toc}}
\makeatother

\usepackage{bm}
\usepackage{amsmath, amsthm, amssymb, amsfonts}
\usepackage{graphicx}
\usepackage{booktabs}
\usepackage{float}
\usepackage{mathrsfs}
\usepackage{mathtools}
\usepackage{color}
\usepackage{enumitem}
\usepackage{tikz-cd}
\usepackage{mathtools}
\usepackage[]{hyperref}
\hypersetup{
    colorlinks=false,
    linkcolor=blue,
    filecolor=magenta,      
    urlcolor=green,
}
\usepackage{cleveref}
\newtheorem{theorem}{Theorem}[section]

\newtheorem{corollary}[theorem]{Corollary}

\newtheorem{proposition}[theorem]{Proposition}

\newtheorem{remark}[theorem]{Remark}

\theoremstyle{definition}
\newtheorem{example}[theorem]{Example}
\newtheorem{definition}[theorem]{Definition}

\usepackage{amsthm}

\usepackage{algorithm}
\usepackage[noend]{algpseudocode}
\algnewcommand\algorithmicinput{\textbf{Input:}}
\algnewcommand\algorithmicoutput{\textbf{Output:}}
\algnewcommand\Input{\item[\algorithmicinput]}%
\algnewcommand\Output{\item[\algorithmicoutput]}

\newcommand{\C}{\mathbb{C}}
\newcommand{\R}{\mathbb{R}}

\renewcommand\bar\overline

\newcommand{\rank}{{\mathrm{rank}}}

\DeclareMathOperator{\Conv}{Conv}

\DeclareMathOperator{\Reg}{Reg}
\DeclareMathOperator{\codim}{codim}

\usepackage{mathtools}
\renewcommand{\perp}{\rotatebox[origin=c]{90}{$\models$}}
\title{Degrees of the Wasserstein Distance to Small Toric Models}
\author{Greg DePaul}
\address{University of California, Davis}
\email{gdepaul@ucdavis.edu}

\author{Serkan Ho\c{s}ten}
\address{San Francisco State University, California}
\email{serkan@sfsu.edu}

\author{Nilava Metya}
\address{Rutgers University, New Jersey}
\email{nilava.metya@rutgers.edu}

\author{Ikenna Nometa}
\address{University of Hawai`i at M$\bar{\text{a}}$noa, Honolulu Hawai`i}
\email{inometa@hawaii.edu}

\newtheorem{conjecture}{Conjecture}

\newcommand{\M}{\mathcal{M}}

\begin{document}

\maketitle
\begin{abstract} The study of the closest point(s) on a statistical model from a given distribution in the probability simplex with respect to a fixed Wasserstein metric gives rise to a polyhedral norm distance optimization problem.   
There are two components to the complexity of
determining the Wasserstein distance from a data point to a model. One is the combinatorial complexity that
is governed by the combinatorics of the Lipschitz polytope of the finite metric to be used. Another is the algebraic complexity, which is governed by the polar degrees of the Zariski closure of the model. We find formulas for the 
polar degrees of rational normal scrolls and graphical models whose underlying graphs are star trees. Also, the polar degrees of the graphical models 
with four binary random variables where
the graphs are a path on four vertices and the four-cycle, as well as for small, no-three-way interaction models, were computed. 
We investigate the algebraic degree of computing the Wasserstein distance to a small subset of these models. It was observed that this algebraic degree is typically smaller than the corresponding polar degree.
\end{abstract}
\section{Introduction}
The probability simplex 
$$ \Delta_{n-1} \, := \, \left\{(p_1, \ldots, p_n) \, : \, \sum_{i=1}^n p_i = 1, \mbox{ and } p_i \geq 0, \, i=1, \ldots,n \right\}$$
consists of probability distributions of a discrete random variable with a state space of size $n$. We  take this state space to be $[n] := \{1, \ldots, n\}$. A {\it statistical model} $\M$ is a subset of $\Delta_{n-1}$ which represents distributions to which a hypothesized unknown distribution $\nu$ belongs. Typically, after collecting data $u=(u_1, \ldots, u_n)$ where $u_i$ is the number of times outcome $i$ is observed, one forms 
the empirical distribution ${\bar \mu} = \frac{1}{N} u$  where $ N = \sum_{i=1}^n u_i$ is the sample size. Note that ${\bar \mu} \in \Delta_{n-1}$. To estimate the unknown distribution $\nu$, a standard approach is to locate $\nu \in \M$, that is a ``closest" point to ${\bar \mu}$. For instance, $\nu$ can be taken to be the maximum likelihood estimator \cite[Chapter 7]{sullivant2018algstat} of ${\bar \mu}$. In this case,
$\nu$ is the point on $\M$ that minimizes the Kullback-Leibler divergence from ${\bar \mu}$ to $\M$. However, Kullback-Leibler divergence is not a metric, and the maximum likelihood estimator does not minimize a true distance function from ${\bar \mu}$ to $\M$.

For the above density estimation problem, one can use a distance minimization approach if the state space $[n]$ is also a metric space. 
A metric on $[n]$ is a collection of nonnegative real numbers $d_{ij}$ for $i,j \in [n]$ such that $d_{ii} = 0$ for all $i \in [n]$, $d_{ij} = d_{ji}$, and the triangle inequality $d_{ik} \leq d_{ij} + d_{jk}$
holds for all $i,j,k \in [n]$. Sometimes, the metric on $[n]$ is written as an $n\times n$ nonnegative symmetric matrix $d=(d_{ij})_{i,j\in [n]}$. Common examples include the discrete metric, the $L_1$ metric, the $L_0$ metric, and the Hamming distance metric.

For two probability distributions $\mu$ and $\nu$ in
$\Delta_{n-1}$, the optimal value $W_d(\mu,\nu)$ of the following linear program is the {\it Wasserstein distance} between $\mu$ and $\nu$ based on the metric $(d_{ij})$:
\begin{equation} \label{eq:Wasserstein program}
\mbox{maximize} \quad \sum_{i=1}^n (\mu_i - \nu_i)x_i \quad \mbox{subject to} \quad |x_i -
x_j| \leq d_{ij} \,\, \mbox{for all} \,\, 1 \leq i < j \leq n.
\end{equation}

The Wasserstein distance $W_d(\mu,\M)$ from $\mu \in \Delta_{n-1}$ to the model $\M$ is the minimum of $W_d(\mu,\nu)$ as
$\nu$ ranges over $\M$: 
\begin{equation} \label{eq: Wasserstein to model}
W_d(\mu, \M) := \min_{\nu\in \M} W_d(\mu, \nu).  
\end{equation}
This has been successfully used to construct a version of Generative Adversarial Networks \cite{WGAN17} where $W_d(\cdot, \M)$ is used as
the loss function. However, for large $n$, computing the exact Wasserstein distance $W_d(\mu, \M)$ is not feasible with the current state of knowledge. 

In this paper our starting point is \cite{Wasserstein2} and \cite{ccelik2021wasserstein} where an algebraic approach for the optimization problem in (\ref{eq: Wasserstein to model})
was developed to handle structured cases when $n$ is small.  We first recall this approach. 

The Wasserstein metric $W_d$ induced  by the finite metric $d$ on $[n]$ defines a 
norm on $\R^{n-1}$. The unit ball of this metric norm is the polytope
$$B = \mathrm{conv} \left\{ \frac{1}{d_{ij}}(e_i - e_j) \, : \, 1 \leq i < j \leq n \right\},$$
where $B$ lies in the hyperplane 
$x_1 + x_2 + \ldots + x_n = 0$.  For $\mu, \nu \in \Delta_{n-1}$, the Wasserstein distance $W_d(\mu,\nu)$ is equal to 
$$ \min \left\{ \lambda \geq 0 \, : \, \nu \in \mu + \lambda B \right\}. $$
The Wasserstein ball $B$ is the dual
of the {\it Lipshitz polytope}
$$ P_d = \left\{ x \in \R^n \, : \, |x_i - x_j| \leq d_{ij}, \, 1 \leq i < j \leq n \right\} \cap \left\{ x_1 + x_2 + \cdots + x_n = 0\right\}.$$
This is essentially the feasible region of the linear program (\ref{eq:Wasserstein program}). Now we can reformulate $W_d(\mu,\M)$, the Wasserstein distance
from $\mu \in \Delta_{n-1}$ to a statistical
model $\M \subset \Delta_{n-1}$:
\begin{equation} \label{eq: inflate the ball}   
W_d(\mu, \M) \, = \, \min \left\{ \lambda \geq 0 \, : \, (\mu + \lambda B) \cap \M \neq \emptyset \right \}.
\end{equation}
In algebraic statistics, the statistical
models considered are typically of 
the form $\M = \Delta_{n-1} \cap X$
where $X$ is a (complex) algebraic variety \cite{sullivant2018algstat}.When $X$ is a projective variety, in the intersection $\Delta_{n-1} \cap X$, we consider the affine variety that is the cone over $X$. 
In this paper, the algebraic variety $X$ is a {\it toric variety} and $\M$ is
a {\it toric model}, otherwise known as a {\it discrete exponential model}. The data
that defines $X$ is encoded as a $d \times n$ integer
matrix 
$$A = \begin{bmatrix} {\bf a}_1 &{\bf a}_2 & \cdots &{\bf a}_n
    \end{bmatrix},$$
that contains $(1,1,\ldots, 1)$ in its row space where ${\bf a}_j$ denotes the $j^\text{th}$ column of $A$, $j\in[n]$. Then $X$ is the Zariski closure
of the image of the monomial map
\begin{align*} (\C^*)^d &\longrightarrow (\C^*)^n \\  ( \theta_1, \ldots, \theta_d) &\longmapsto 
(p_1 = \theta^{\bf a_1}, \ldots, p_n = \theta^{\bf a_n}).
\end{align*}
Here $\theta^{{\bf a}_j}=\theta_1^{a_{1j}}\theta_2^{a_{2j}}\cdots \theta_d^{a_{dj}}$ for $j=1,2,\ldots, n$. The dimension of the projective toric variety $X$ and the discrete exponential model $\M$ both equal $\rank(A)-1$. 

\begin{example} \label{ex:twisted}
Consider tossing a biased coin three times and let $Y$ be the number of heads observed. If the probability of getting a head is $p$, so that tails come up with probability $q=1-p$,  then $\left(q^3, 3q^2p, 3qp^2, p^3\right) \in \Delta_3$ parametrizes all possible probability distributions of $Y$. The corresponding toric
variety $X$ is a curve that is isomorphic to the 
twisted cubic, and it is given by the equations
 \begin{align*}
    g_1 &= 3p_0p_2 - p_1^2 = 0\\
    g_2 &= 3p_1p_3-p_2^2 = 0\\
    g_3 &= 9p_0p_3-p_1p_2 = 0\\
    g_4 &= p_0+p_1+p_2+p_3 - 1 =0.
\end{align*} 
We use the discrete metric on $[4]=\{1,2,3,4\}$. This is the metric associated with $(d_{12}, d_{13}, d_{14}, d_{23}, d_{24}, d_{34}) = (1, 1,\cdots , 1)$. With 
this, the Wasserstein ball is a cuboctahedron with
$12$ vertices, $24$ edges, and $14$ facets. 
    \Cref{fig:wassex} is a picture of the simplex $\Delta_3$ together with 
    the twisted cubic (orange) curve and the Wasserstein ball (blue) centered at  
     $\mu=\left(\frac{1}{4},\frac{1}{5}, \frac{1}{6}, \frac{23}{60}\right)$. Slowly the radius $\lambda$ of the  Wasserstein ball around $\mu$ is decreased until it touches the curve. The snapshot shows $\lambda \approx 0.268$ when the ball just touches the curve at the red point $\nu\approx (0.02,0.16,0.44,0.38)$, and the point lies on an edge of the Wasserstein ball.
\end{example}

\begin{figure}
    \centering
    \includegraphics[width=2.5in]{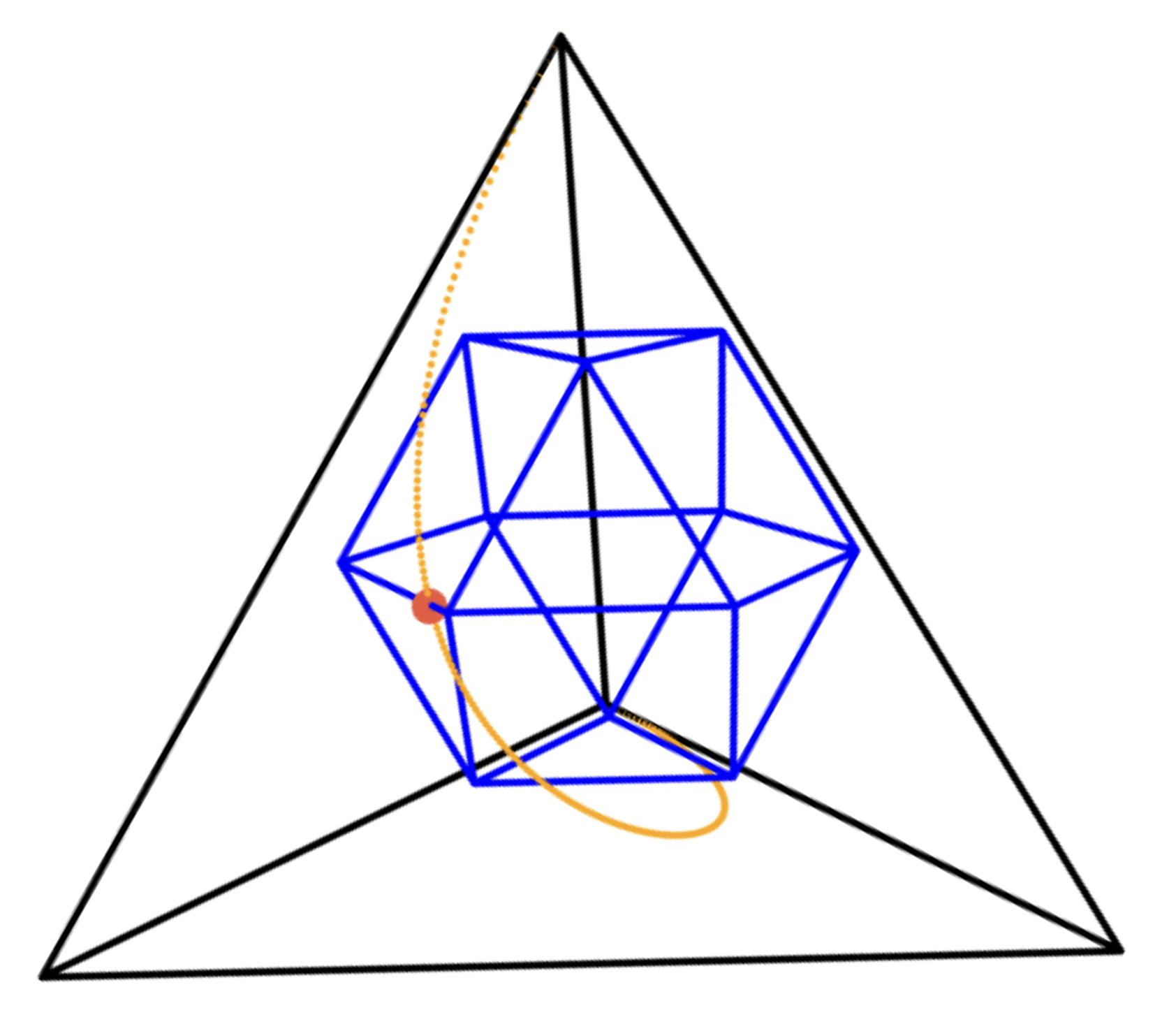}
    \caption{Wasserstein ball for the discrete metric touching the twisted cubic at an edge.}
    \label{fig:wassex}
\end{figure}

Understanding the geometry of the distance
minimization problem from a point to an algebraic variety is part of {\it metric algebraic geometry} \cite{BKSmetricAlgGeo, WeinsteinThesis}. This problem has been studied extensively when the distance metric is the usual Euclidean metric leading to the {\it Euclidean distance degree} of an algebraic variety; see \cite{draisma2016euclidean}, \cite{Helmer_Sturmfels_2018}. We wish to continue 
studying this problem in the case of the Wasserstein metric, first considered in \cite{ccelik2021wasserstein}.

Geometrically, it is not difficult to imagine what
needs to be done to find $W_d(\mu,\M)$ using
(\ref{eq: inflate the ball}). One scales
the Wasserstein unit ball $B$ centered at $\mu$ until it touches the model $\M$ for the very first time. As long as $\mu$ and $\M$ are generic with
respect to $B$, there will be a unique
intersection point, and it will be in the relative interior of one of the faces $F$ of $B$; see
\cite[Proposition 4]{ccelik2021wasserstein}. We can formulate the computation of this point as 
a linear optimization problem over an algebraic variety. We let $\mathcal{L}_F$ be the linear subspace generated by the vertices of the face $F$
of $B$ and let $\ell_F$ be any linear functional
that attains its maximum over $B$ at $F$. Then 
the optimal solution to 
\begin{equation} \label{eq: linear opt}
\mbox{minimize} \,\, \ell_F(\nu) \,\,\,
\mbox{subject to} \,\,  \nu \in (\mu + \mathcal{L}_F) \cap \M
\end{equation}
is the point we are looking for. From an algebraic
perspective, the crucial observation is that the optimal solution to (\ref{eq: linear opt}) is
one of the finitely many complex critical
points of $\ell_F$ over the variety $(\mu + \mathcal{L}_F) \cap X$. Therefore, the algebraic
complexity of finding the point on $\M$ that 
minimizes the Wasserstein distance from $\mu$ with
respect to the face $F$ is the number of these
complex critical points. 
\begin{definition} Let $F$ be a face of the 
Wasserstein ball $B$. Moreover, let $\mathcal{L}_F$ be the linear space spanned by
the vertices of $F$ and let $\ell_F$ be a generic linear functional that attains its maximum over $B$ in $F$. 
The {\it Wasserstein degree} $w(X,F)$ of the algebraic variety $X$ with respect to the face $F$ is 
the number of complex critical points of $\ell_F$
over $(\mu + \mathcal{L}_F) \cap X$
for a generic point
$\mu \in \Delta_{n-1}$. 
\end{definition}

We note that $w(X,F)$ indeed does not depend on the choice of the point $\mu$. This will follow from the algorithm we present in Section \ref{section3} to compute $w(X,F)$; see Proposition \ref{prop:Wassserstein-degree}.

The Wasserstein degrees $w(X,F)$ as $F$ ranges
over all faces of the Wasserstein ball $B$ 
are bounded above by the {\it polar degrees} of $X$. We will define polar degrees in the next section. \begin{theorem} \cite[Theorem 13]{ccelik2021wasserstein} \label{thm: polar-vs-Wasserstein} The Wasserstein degree $w(X,F)$ where\\ $\codim(F)=r+1$ is at most $\mu_i(X)$, the $i$th polar degree of $X$ where $i=\dim(X)-r$.     
\end{theorem}
The next section will investigate various toric models and compute their polar degrees. We find formulas for the 
polar degrees of rational normal scrolls and graphical models whose underlying graphs are star trees. We also compute the polar degrees of the graphical models 
with four binary random variables where
the graphs are a path on four vertices and the four-cycle, as well as for small, no-three-way interaction models. 
Then, we investigate the Wasserstein degree $w(X,F)$ for a small subset of these models. We observe that, typically, this degree is smaller than the corresponding polar degree. 

\section{Polar Degrees}
In this section, we introduce the polar degrees of a projective variety. The notation used is adopted from \cite[Chapter 4]{BKSmetricAlgGeo}. 
There are formulas for the polar degrees of independence models, which are toric 
varieties that are products of projective spaces \cite[Theorem 14 and Corollary 15]{ccelik2021wasserstein}. 
We will provide formulas for the polar degrees of Hirzebruch surfaces, more generally, for rational normal scrolls.
We will also give formulas for the polar degrees of discrete graphical models whose underlying graph is a star tree. Moreover, we will report our computations of polar degrees for graphical models coming from graphs that are paths. These models are decomposable graphical models, a generalization of independence models. 
We will also treat a non-decomposable graphical model that comes from a binary four-cycle. Finally, we will report the polar degrees of some no-three-way interaction models. These are hierarchical log-linear models that are not decomposable. 

\subsection{Polar Degrees}

By Theorem \ref{thm: polar-vs-Wasserstein}, the Wasserstein degree of 
a model, $\M$ with respect to a face $F$ of the Wasserstein ball, is bounded above by the corresponding polar degree of the associated variety $X$. Polar degrees appear in similar applications for metric optimization problems. For example, Euclidean distance degrees, which measure the complexity of Euclidean distance optimization problem to $X$, are calculated using polar degrees  (see \cite{draisma2016euclidean},\cite{helmer2019polar},\cite{Helmer_Sturmfels_2018}). 

Polar degrees can be defined in terms of Schubert varieties and the Gauss map, multidegree of conormal varieties, or via non-transversal intersections with generic linear subspaces (see \cite[Chapter 4]{BKSmetricAlgGeo}). We will use the latter approach for a first definition. 

A projective linear subspace $V\subseteq \mathbb{P}^n$ is said to intersect a variety $X$ non-transversally at a point $p\in \Reg(X)$ if $p\in V$ and $\dim\left(V+T_{p}X\right)<n$. Here $V+T_{p}X$ denotes the projective span of $V$ and the tangent space of $X$ at $p$, while $\Reg(X)$ denotes the regular (nonsingular) locus of the variety $X$.

 \begin{definition}[Polar variety \cite{BKSmetricAlgGeo}]\label{def:polarVariety}
    The \textit{polar variety} of an irreducible variety $X\subset \mathbb P^n$ with respect to the linear subspace $V\subset \mathbb P^n$ is
	\begin{equation}
		P(X,V):=\overline{\{p \in Reg(X)\setminus V\,  : \,  p+V \text{ intersects } X \text{ at } p \text{ non-transversally}\}}.
		\label{eq::polarvariety}
	\end{equation}
\end{definition}	
    
    Let $i\in\{0,1,\ldots,\dim(X)\}$. If $V$ is generic with $\dim(V) = \codim (X)-2+i$, then the degree, $\deg(P(X, V))$,  of the polar variety $P(X, V)$ is independent of $V$. 
\begin{definition}[Polar degree]\label{def:polardeg}
	The integer $\mu_i(X)$, for $i=\dim (V) -\codim(X)+2$,  is called the $i^{th}$ 
        {\it polar degree} of $X$, where 
        \begin{equation}
            \mu_i(X):=\deg(P(X,V)).
            \label{eq:polardegree}
        \end{equation}
		
\end{definition}


It is easy to prove the following effective 
conditions for non-transversal intersections. 
\begin{proposition}\label{prop:Nontransversality}
A projective linear subspace $V\subseteq \mathbb{P}^n$ intersects a variety $X$ non-transversally at a point $p\in \Reg(X)$ if $p\in V$ and 
\begin{equation}\label{eq:NewnonTransver}
    \dim(V) - \codim(X) < \dim(T_pX \cap V). 
\end{equation}
\end{proposition}

Notice that \Cref{prop:Nontransversality} can be used to assign a natural bound on what subspaces $V$ we need to consider when computing polar degrees of $X$.  When $\dim(V)$ is small enough, the non-transversality condition will hold for all points of $X$, hence $P(X,V) = X$. 
Therefore, we use those $V$ which satisfy  $\dim(V) = \codim(X) - 2 + i$. In other words, we consider 
$V$ where  $\codim(X) - 2  \leq  \dim(V) \leq n-2$.

   \begin{example}[Rational Normal Quartic Curve]\label{eg:polarDcomputation}
    Consider the rational map
    \begin{align*}
    \phi: \mathbb P^1 &\longrightarrow \mathbb P^4  \text{ given by }\\
    \phi([s:t]) & = \left[\binom{4}{0}s^4 :\binom{4}{1}s^3t : \binom{4}{2}s^2t^2 : \binom{4}{3}st^3 : \binom{4}{4}t^4 \right] = \left[s^4 :4s^3t : 6s^2t^2 : 4st^3 : t^4 \right].
    \end{align*}
    Let $X= \text{im}(\phi)$. Then $X$ has dimension $1$ and codimension $3$. The model $\M = X \cap \Delta_4$ is the set of probability distibutions of the random variable $Z$ that counts the number of heads 
    of a biased coin flipped four times in a row where the probability of observing head is $s$. To compute the polar degrees, we only need to consider linear subspaces $V\subset \mathbb P^4$ such that 
    $$\codim(X)-2\leq \dim(V)\leq n-2 \implies 1\leq \dim(V) \leq 2. $$
    
    \begin{itemize}
    \item If $V$ has dimension $1$, from the relation $\dim(V)=\codim(X)-2+i,$ we have $i=0$, and by  \cite[Theorem 4.14(b)]{BKSmetricAlgGeo} \textbf{$$\mu_0 (X)= \deg(X)=4.$$}
    
    \item If $\dim(V)=2$, then $i=1$. The tangent space $T_{p}X$ where $p\in \Reg(X)\setminus V$ is spanned by the Jacobian of the matrix
    $$
    \begin{bmatrix}
        4s^3 & 12s^2t & 12st^2 & 4t^3 & 0 \\
        0 & 4s^3 & 12s^2t & 12st^2 & 4t^3
    \end{bmatrix},
    $$ which can be rescaled to 
    $$
    \begin{bmatrix}
        s^3 & 3s^2t & 3st^2 & t^3 & 0 \\
        0 & s^3 & 3s^2t & 3st^2 & t^3
    \end{bmatrix}.
    $$
    $V$ is the span of three generic points. For example, we can choose 
    $$V = \text{span} \{[1:3:7:1:1],  [2:2:1:3:1], [1:1:1:3:2] \}.$$ For a non-transversal intersection of $p + V$ with $X$, we want 
    $$\dim(V)-\codim(X)=2-3=-1<\dim(T_{p}X+V).$$ 
    This is equivalent to requiring the determinant of the matrix
    $$
    \begin{bmatrix}
        s^3 & 3s^2t & 3st^2 & t^3 & 0 \\
        0 & s^3 & 3s^2t & 3st^2 & t^3\\
        1 & 3 & 7 & 1 & 1 \\
        2 & 2 & 1 & 3 & 1 \\
        1 & 1 & 1 & 3 & 2
    \end{bmatrix}
    $$
    to vanish. That is, if $20s^6 - 27s^5t - 33s^4t^2 + 127s^3t^3 - 96s^2t^4 + 18st^5 - 2t^6=0$. We conclude, therefore, that  $\mu_1(X)=6$.
\end{itemize}
 \end{example}

A second way of defining polar degrees of a projective variety $X$  is by using the multidegree of the conormal variety of $X$.
The dual projective space $(\mathbb{P}^n)^\vee$ parametrizes hyperplanes 
$\{x\in\mathbb{P}^n:\sum_{i=0}^nu_ix_i=0\}$ in $\mathbb{P}^n$ where $u$ is the corresponding point  in $(\mathbb{P}^n)^\vee$. 
The conormal variety $N_X$ of $X$ is defined as follows.

\begin{definition}[Conormal Variety]
Given a variety $X\subseteq \mathbb{P}^n$, the cornomal variety of $X$, denoted $N_X$ is
\begin{equation}\label{eq:conormalVar}
    N_X :=\overline{\{(x,u)\in \mathbb{P}^n\times(\mathbb{P}^n)^\vee \, : \, x\in\Reg(X) \text{ and } u \text{ is tangent to } X \text{ at } x\}}.
\end{equation}
\end{definition}

Recall that $u$ is tangent to $X$ at $x$ if and only if $H_u\supset T_xX$, where $H_u$ is the hyperplane in $\mathbb P^n$ associated with the point $u \in (\mathbb P^n)^\vee$.

The dimension of the conormal variety $N_X$ is $n-1$, and the dual variety $X^\vee$ of $X$ is the image of $N_X$ under the projection to 
$(\mathbb{P}^n)^\vee$ (see \cite[Chapter 4]{BKSmetricAlgGeo}). Now, for two generic linear
subspaces $L_1$ and $L_2$ where $\dim(L_1) = n+1-j$ and $\dim(L_2) = j$ we define
$$ \delta_j \, := \, |N_X \cap (L_1 \times L_2)|. $$
The multidegree of $N_X$ is the bivariate polynomial
$$ \sum_{j=1}^n \delta_j s^{n+1-j}t^j.$$ 

It is known that $\delta_j(X) = 0$ whenever 
$j < \codim(X^\vee)$ or $j > \dim(X) + 1$. Moreover, for $ j = \codim(X^\vee)$ and $j = \dim(X)+1$, $\delta_j(X)= \deg(X^\vee)$  and 
$\delta_j(X) = \deg(X)$, respectively.
\begin{theorem} \cite[Theorem 4.16]{BKSmetricAlgGeo}
The multidegree of $N_X$ of $X$ agrees with the polar degrees of $X$. More precisely, $\delta_j(X) = \mu_i(X)$ where 
$i = \dim(X) + 1 -j$.    
\end{theorem}

Methods for computing multidegrees and polar degrees have been implemented in \texttt{ Macaulay2} \cite{M2}. In particular, in the \texttt{Resultants}  package \cite{ResultantsArticle, ResultantsSource}, one can use the \texttt{multidegree} function to compute and read off polar degrees. The following lines of code show how polar degrees of the variety in \Cref{eg:polarDcomputation} can be computed in \texttt{Macaulay2}.
\begin{lstlisting}
i1 : needsPackage "Resultants"
o1 = Resultants
o1 : Package
i2 : I = kernel map(QQ[s, t], QQ[x1,x2,x3,x4,x5], {s^4, 4*s^3*t, 6*s^2*t    ^2, 4*s*t^3, t^4});
o2 : Ideal of QQ[x1, x2, x3, x4, x5]
i3 : X = variety I
o3 = X
o3 : ProjectiveVariety
i4 : dim X, codim X
o4 = (1, 3)
o4 : Sequence
i5 : multidegree conormalVariety I
       4       3 2
o5 = 6T T  + 4T T
       0 1     0 1
o5 : ZZ[T ..T ]
         0   1
\end{lstlisting}
%
Also, the \texttt{Macaulay2} package \texttt{ToricInvariants} \cite{helmer2019polar, ToricInvariantsSource} can
be used to compute polar degrees of toric models using
the function \texttt{polarDegrees} (see also \cite{Helmer_Sturmfels_2018}). 
Some of the code we used for these computations can be found in \cite{GitHub_site}. 

\subsection{Rational Normal Scrolls}

A rational normal scroll is a toric variety associated with a $(d+1)\times n$ matrix $A$ of the following form:
\begin{equation}
A = \left[\begin{array}{ccccccccccccc}

1 & \cdots & 1 & 1 & \cdots & 1 & \cdots & 1 &
\cdots & 1 & 1 & \cdots & 1 \\
1 & \cdots & 1 & 0 & \cdots & 0 & \cdots & 0 & \cdots & 0 & 0 & \cdots & 0 \\
        0 & \cdots & 0 & 1 & \cdots & 1 & \cdots & 0 & \cdots & 0 & 0 & \cdots & 0 \\
        \vdots & \vdots  &  \vdots  & \vdots & \vdots &  & \vdots &  & \vdots &  \vdots & 0 & \vdots &  \vdots  \\
        0 & \cdots & 0 & 0 & \cdots & 0 & \cdots & 1 & \cdots & 1 & 0 & \cdots & 0 \\
        0 & 1\cdots & n_1 & 0 & 1\cdots &n_2  & \cdots & 0 & \cdots & 0 & 1 & \cdots & n_{d} \\
\end{array}\right].
\end{equation}

Here $n=n_1+n_2+\cdots+n_{d} + d$ where $n_1, \ldots, n_{d}$ are positive integers. 
Let $S=S(n_1,n_2,\ldots,n_{d})$ be the rational normal scroll in $\mathbb{P}^{n-1}$ corresponding to the sequence of positive integers
$n_1, n_2, \ldots, n_d$. Its defining ideal is given by the $2$-minors of $M$, where
\begin{equation}
M=\left[\begin{array}{c|c|c|c}
     M_{n_1} & M_{n_2} &\ldots & M_{n_{d}}  \\
\end{array}
\right]
\text{ and } M_{n_j} =
\left[\begin{array}{ccc}
     x_{j,0} & \cdots & x_{j,n_j-1}  \\
      x_{j,1} & \cdots & x_{j,n_j}  
\end{array}
\right].
\end{equation}
This description is due to Petrovi\'c \cite{petrovic2008universal}, where more details about this characterization of rational normal scrolls can be found. By definition $\dim(S) = d$ and its degree
is equal to $N = \sum_{j=1}^d n_j$.

In particular, when $d=2$, we recover Hirzebruch
surfaces $S(a,b)$ where we assume $a \leq b$.
The corresponding polytope is the quadrangle $P=\Conv\{(0,0), (a,0), (0,1), (b,1)\}$; see \Cref{fig:abpolytope}. The defining matrix $A$ 
is explicitly
\begin{equation}
 A = \left[
     \begin{array}{ccccccccc}
         1 & 1  & \cdots & 1& 1 & 1 & 1 &\cdots &1 \\
          0 & 1 & 2& \cdots &a & 0 & 1 & \cdots &b \\
          0 & 0 & 0&
         \cdots & 0& 1 & 1 & \cdots &1 \\
     \end{array}\right]
  \label{Eq:HS}
\end{equation}
which gives the monomial parametrization
$(s, t_1, t_2) \mapsto (s, st_1\ldots,st_1^a, st_2,st_1t_2, \ldots,st_1^bt_2)$.
\begin{figure}
    \centering
    \includegraphics[scale=0.7]{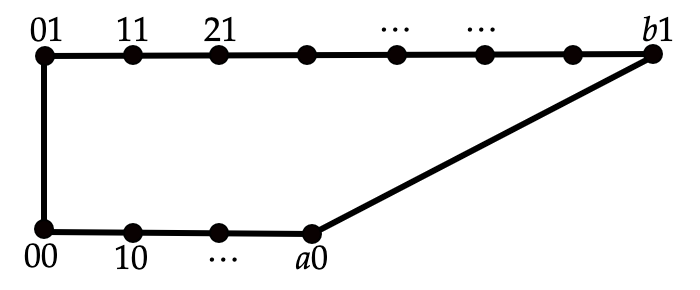}
    \caption{Polytope for the Hirzebruch surface $S(a,b)$.}
    \label{fig:abpolytope}
\end{figure}%

\begin{theorem}
 \label{thm:polarDegRationNormal}
      Let $S(n_1,\ldots, n_{d}) \subset \mathbb{P}^{n-1}$ be a rational normal scroll and let $N=~\sum_{j=1}^{d} n_j$. The 
      polar degrees of $S(n_1, \ldots, n_{d})$ are  
    \begin{equation}
        \delta_j =\begin{cases}
        N, & \text{ if }j=d-1,d+1 \\
        2(N-1), & \text{ if }j=d \\
        0 & \text{ otherwise.}
        \\
    \end{cases}
  \end{equation}
 \end{theorem}
 \begin{proof}
 The rational normal scroll $S = S(n_1, \ldots, n_d)$ is equal to the intersection of the Segre embedding $X_{1,N-1}$ of $\mathbb{P}^1 \times \mathbb{P}^{N-1}$ in $\mathbb{P}^{2N-1}$
  with a linear subspace $\Pi$ of dimension $(n-1)$. Therefore $S^\vee$
  is obtained by projecting $X_{1,N-1}^\vee$ from $\Pi^\vee$. It is 
  known that $X_{1,N-1}^\vee = X_{1,N-1}$.
  Since the degree of $X_{1,N-1}$ is $N$
  we conclude that $\deg(S^\vee) = N$. 
  This also means that 
  $\codim(S^\vee) = (n-1)-N= d-1$. Therefore $\delta_{d-1} = N$.
  Because $\deg(S)$ is also $N$, for
  $\dim(S) + 1 = d+1$ we get $\delta_{d+1} = N$. Furthermore, according to \cite[Theorem 4.15]{BKSmetricAlgGeo} $\mu_1(S^\vee) = \mu_1(S) = \delta_d(S)$ which means that $\mu_1(S^\vee) = \delta_N(S^\vee)$. Again, since $S^\vee$ is obtained by projecting $X_{1,N-1}^\vee = X_{1,N-1}$ from $\Pi^\vee$, we need to compute $\delta_N(X_{1,N-1})$. We use the formula from \cite[Corollary 16]{ccelik2021wasserstein} to compute this polar degree to be $2(N-1)$.
 \end{proof}



\begin{corollary}
\label{cor:Hirzebruch}
    Let $S = S(a,b) \subset \mathbb P^{b+a+1}$ be a parameterized Hirzebruch toric surface determined by $a < b$. Then the 
    polar degrees of $S(a,b)$ are 
    
    \[\delta_j(S) =\begin{cases}
        a+b, & \text{ if } j=1,3 \\
        2(a+b-1), & \text{ if }j=2 \\
        0 & \text{ otherwise.}
        \\
    \end{cases}
    \] 
    \end{corollary}

\subsection{Small graphical models}
Let $G=(V, E)$ be a graph. $G$ is chordal if every induced cycle of $G$ is a $3-cycle$. A graph $G$ is decomposable if and only if $G$ is chordal \cite{dirac1961rigid}.
In this section, we compute the polar degrees of some small graphical models 
on discrete random variables. The models we consider are those based on graphs that are star graphs and paths
with at most four vertices. These are 
{\it decomposable} models. We note that 
independence models that were treated 
in \cite{ccelik2021wasserstein} are also decomposable graphical models corresponding to graphs with no edges. We are continuing the program of understanding the polar degrees and Wasserstein degrees of 
decomposable graphical models. Furthermore, we will also consider the 
smallest non-decomposable graphical model given by four binary random variables and the four-cycle graph. Finally, we will also compute the polar degrees of small
no-three-way interaction models. 

Let $G = (V,E)$ be a simple undirected graph on $n$ vertices
and let $Z_i$, $i=1, \ldots, n$ be discrete random variables with state 
space $[d_i]$. We let $\mathbf{d} = (d_1, \ldots, d_n)$. For each maximal clique
$C$ in $G$ with vertex set $V(C)$, 
we introduce parameters $\theta^C_I$
where $I \in \prod_{i\in V(C)} [d_i]$. We denote the set of all maximal
cliques in $G$ by $\mathcal{C}$.
Now we let $X_{G,\mathbf{d}}$ be the toric variety
defined by the monomial parametrization
$$ p_I = \prod_{C \in \mathcal{C}} \theta^C_{I(C)}$$
where $I \in \prod_{i}^n [d_i]$.
When $I = (i_1, \ldots, i_n)$ and 
$V(C)=\{j_1, \ldots, j_k\}$ 
we let $I(C) = (i_{j_1}, i_{j_2}, \ldots, i_{j_k})$. The 
toric variety $X_{G, \mathbf{d}}$
is a projective variety in $\mathbb{P}^{D-1}$ where $D = \prod_{i=1}^nd_i$. The graphical model 
$\M_{G,\mathbf{d}}$ is equal to 
$X_{G, \mathbf{d}} \cap \Delta_{D-1}$.
If $G$ consists of $n$ isolated vertices $\M_{G,\mathbf{d}}$ is called a {\it complete independence model}. 
 
\subsubsection{Polar degrees of star trees}
A star tree is a graph that is a tree 
with a single internal vertex as in \Cref{fig:GeneralStarGraph}. 
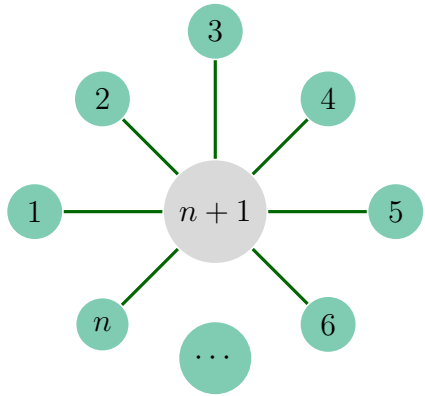
\begin{figure}[ht!]
    \centering
    \begin{tikzpicture}
    [scale=.3,auto=left,very thick,-]
\node [circle, fill=green!60!blue!50!white] (nA1) at (-8,0) {$1$};
\node [circle, fill=green!60!blue!50!white] (nA2) at (-5,5)  {$2$};
\node [circle, fill=green!60!blue!50!white] (nA3) at (0,8)  {$3$};
\node [circle, fill=green!60!blue!50!white] (nA4) at (5,5)  {$4$};
\node [circle, fill=green!60!blue!50!white] (nA5) at (8,0)  {$5$};
\node [circle, fill=green!60!blue!50!white] (nA6) at (5,-5)  {$6$};
\node [circle, fill=green!60!blue!50!white] (nA7) at (0,-6.5)  {$\cdots$};
\node [circle, fill=green!60!blue!50!white] (nAD) at (-5,-5)  {$n$};
\node [circle, fill=gray!30](nB) at (0,0)  {$n+1$};
\draw (nB) -- (nA1) [green!40!black] node [midway, right, black ] {};
\draw (nB) -- (nA2) [green!40!black] node [midway, right, black ] {};
\draw (nB) -- (nA3) [green!40!black] node [midway, right, black ] {};
\draw (nB) -- (nA4) [green!40!black] node [midway, right, black ] {};
\draw (nB) -- (nA5) [green!40!black] node [midway, right, black ] {};
\draw (nB) -- (nA6) [green!40!black] node [midway, right, black ] {};
\draw (nB) -- (nAD) [green!40!black] node [midway, right, black ] {};
\end{tikzpicture}
    \caption{A star tree with $n$ leaves.}
    \label{fig:GeneralStarGraph}
\end{figure}
We label the leaves of the tree by $1, \ldots, n$ and let the label of the internal vertex be $n+1$. The resulting 
model $\M_{G, \mathbf{d}}$ is a {\it conditional independence model} corresponding to the conditional independence statement
$Z_1 \, \perp \, Z_2  \, \perp \cdots \, \perp \, Z_n \mid Z_{n+1}$
for all $1 \leq i < j \leq n$. 
The equations defining $X_{G, \mathbf{d}}$ in this case are easy to 
describe:

\begin{proposition} \label{prop:tree-equations}
Let $G$ be a star tree with $n$ leaves
and let $\mathbf{d} = (d_1, \ldots, d_n, K)$. For each $1 \leq k \leq K$
let $I_k \subset \C[p_{i_1 \cdots i_n k}]$ be the ideal defining 
the complete independence model on $n$
random variables $Z_1, \ldots, Z_n$ with state spaces
$[d_i]$, $i=1, \ldots, n$.
Then 
$$I(X_{G,\mathbf{d}}) = I_1 + I_2 + \cdots + I_K.$$
    
\end{proposition}
We note that the conditional independence models defined by $I_k$ for each $k$ are isomorphic varieties. 

\begin{theorem} \label{thm:star-tree}
Let $G$ be a star tree with $n$ 
leaves and let $\mathbf{d} = (d_1, \ldots, d_n, K)$. Let $Y_{n, \mathbf{d'}}$ be the complete independence model on $n$ random variables with state spaces 
given by $\mathbf{d'} = (d_1, \ldots, d_n)$. If $h(s,t)$ is the multidegree
of $N_{Y_{n, \mathbf{d'}}}$ then
the multidegree of $N_{X_{G, \mathbf{d}}}$
is $h(s,t)^K$.
\end{theorem}
\begin{proof}
    The proof follows from a slightly more general statement as in Proposition \ref{prop:product-lemma} below.
    \end{proof}

\begin{example} \label{eg:starTreeOnTwoLeaves}
 Let $G$ be the star tree with two leaves 
 where the two random variables corresponding to them are binary. We consider the cases where the random
 variable for the internal node is binary and ternary. Here, $Y_{2,(2,2)}$
 is the Segre embedding of $\mathbb{P}^1 \times \mathbb{P}^1$. The multidegree of the conormal variety of $Y_{2,(2,2)}$
 is 
 $$ 2s^3t + 2s^2t^2 + 2st^3.$$
 The multidegree of the conormal variety
 of $X_{G,(2,2,2)}$ is 
 $$ 4s^6t^2 + 8s^5t^3 + 12s^4t^4 + 8 s^3t^5 + 4 s^2t^6,$$ and that of  $X_{G,(2,3,2)}$ is
 $$8s^9t^3 + 24s^8t^4 + 48s^7t^5 + 56 s^6t^6 + 48 s^5t^7 + 24 s^4t^8 + 8 s^3t^9. $$
 We observe that the last two polynomials are the second and third powers of the first polynomial, as predicted by the above theorem. 
\end{example}

\begin{example}
 Let $G$ be the star tree with three leaves
 where all four random variables are binary. The multidegree of the conormal variety of the complete independence model on three binary variables is
$$ 4 s^7t + 12s^6t^2 + 12 s^5t^3 + 6s^4t^4.$$ 
And the multidegree of the conormal variety of the graphical model $X_{G,(2,2,2,2)}$ 
is 
$$16s^{14}t^2 + 96s^{13}t^3 + 240s^{12}t^4 + 336 s^{11}t^5 + 288 s^{10}t^6 + 144s^9t^7 + 36s^8t^8.$$
\end{example}

\begin{proposition} \label{prop:product-lemma} Let $X \subset \mathbb{P}^n$ and
$Y \subset \mathbb{P}^m$ be two irreducible projective varieties with
defining ideals $I \subset \C[x_0, \ldots,x_n]$ and $J \subset \C[y_0, \ldots, y_m]$, respectively. Let $Z \subset \mathbb{P}^{n+m+1}$ be the irreducible variety defined by $I + J$ 
in $\C[x_0, \ldots, x_n, y_0, \ldots,y_m]$.
Then, the multidegree of $N_Z$ is the product of the multidegrees of $N_X$ and $N_Y$. 
\end{proposition}
\begin{proof}
We first observe that 
$N_Z$ is the Zariski closure of points
$\left( (x,y), (u,v) \right)$
such that $x \in \Reg(X)$, $y \in \Reg(Y)$, and $u \in (\mathbb{P}^n)^\vee$ that is tangent to $X$ at $x$ and
$v \in (\mathbb{P}^m)^\vee$ that is tangent to $Y$ at $y$. 
 The polar degree $\delta_k$ for $Z$
 is equal to $|N_Z \cap (L_{n+m+2-k} \times L_k')|$ where $L_{n+m+2-k} \subset \mathbb{P}^{n+m+1}$
 and $L_k'\subset (\mathbb{P}^\vee)^{n+m+1}$
 are generic subspaces of dimension $n+m+2-k$ and $k$, respectively. We can choose
 $L_{n+m+2-k} = \tilde{L}_{n+1-i} \times \tilde{L}_{m+1-j}$ with $i+j=k$ and $\tilde{L}_{n+1-i}$ and $\tilde{L}_{m+1-j}$ to be generic in $\mathbb{P}^n$ and $\mathbb{P}^m$, respectively. 
 Similarly, $L_k' = \tilde{L}_i' \times \tilde{L}_j'$. Now
 $$\delta_k(Z) = |N_Z \cap (L_{n+m+2-k} \times L_k')| = \sum_{i+j=k} |N_X \cap (\tilde{L}_{n+1-i} \times \tilde{L}_i')| \cdot |N_Y \cap (\tilde{L}_{m+1-j} \times \tilde{L}_j')|, $$
 and the last sum is equal to 
 $\sum_{i+j=k} \delta_i(X)\delta_j(Y)$.
 From this, the result follows. 
\end{proof}

\begin{remark}
Since the multidegree of $N_{Y_{n, \mathbf{d'}}}$, in other words, polar
degrees of the complete independence models $Y_{n, \mathbf{d'}}$ have been explicitly computed (see \cite{ccelik2021wasserstein} and \cite{sodomaco2020distance}), we can also give explicit formulas for the graphical models on any
star tree by Theorem \ref{thm:star-tree}.
\end{remark}

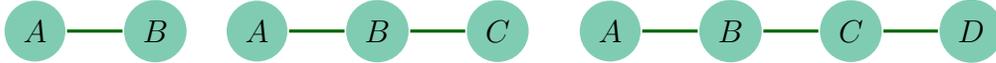
\begin{figure}
     \centering
   \begin{minipage}{2.5cm}
         \centering
         \begin{tikzpicture}
  [scale=.4,auto=left,very thick,-]
\node [circle, fill=green!60!blue!50!white] (nA) at (-4,0) {$A$};
\node [circle, fill=green!60!blue!50!white](nB) at (0,0)  {$B$};
\draw (nA) -- (nB) [green!40!black]  node [midway, left, black ] {};
\end{tikzpicture}
     \end{minipage}%
    \begin{minipage}{5cm}
         \centering
         \begin{tikzpicture}
  [scale=.4,auto=left,very thick,-]
\node [circle, fill=green!60!blue!50!white] (nA) at (-4,0) {$A$};
\node [circle, fill=green!60!blue!50!white] (nC) at (4,0)  {$C$};
\node [circle, fill=green!60!blue!50!white](nB) at (0,0)  {$B$};
\draw (nA) -- (nB) [green!40!black]  node [midway, left, black  ] {};
\draw (nB) -- (nC) [green!40!black] node [midway, right, black ] {};
\end{tikzpicture}
     \end{minipage}%
     \begin{minipage}{6cm}
         \centering
         \begin{tikzpicture}
  [scale=.4,auto=left,very thick,-]
\node [circle, fill=green!60!blue!50!white] (nA) at (-4,0) {$A$};
\node [circle, fill=green!60!blue!50!white] (nC) at (4,0)  {$C$};
\node [circle, fill=green!60!blue!50!white] (nD) at (8,0)  {$D$};
\node [circle, fill=green!60!blue!50!white](nB) at (0,0)  {$B$};
\draw (nB) -- (nA) [green!40!black]  node [midway, left, black  ] {};
\draw (nB) -- (nC) [green!40!black] node [midway, right, black ] {};
\draw (nC) -- (nD) [green!40!black] node [midway, right, black ] {};
\end{tikzpicture}
     \end{minipage}
     \caption{Small path graphs on $2,3,$ and $4$ vertices.}
     \label{fig:BinPathGraphs}
\end{figure}

\subsubsection{Polar degrees of the binary $4$-path}
Another example we looked at is the graphical model
where $G$ is a path with $n$ vertices. When $n=3$, 
$G$ is a star tree with two leaves. We have discussed
this model above in \Cref{eg:starTreeOnTwoLeaves}. The next case is the path with 
four vertices, and we were able to compute the polar
degrees of $X_{G,(2,2,2,2)} \subset \mathbb{P}^{15}$. In this case, the 
toric variety is defined by the parametrization
$$ p_{ijk\ell} = a_{ij} b_{jk}c_{k\ell} \mbox{ for } 1 \leq i,j,k,\ell \leq 2.$$
This variety has a dimension of $7$ and a degree of $34$. 
The polar degrees are given by 
$$ 8s^{14}t^2 + 56s^{13}t^3 + 152s^{12}t^4 + 344s^{11}t^5 + 280 s^{10}t^6 + 136 s^9t^7 + 34 s^8t^8.$$



\subsubsection{Polar degrees of the binary $4$-cycle}
In all the examples we have looked at so far, we dealt
with toric varieties associated with graphical models
that are {\it decomposable} \cite{sullivant2018algstat} \cite{lauritzen}.
Our next example is the smallest graphical model that is not decomposable. The underlying graph $G$ 
is a cycle with four vertices. We take all four
random variables to be binary. The toric
variety $X_{G, (2,2,2,2)}$ has the parametrization
\[p_{ijkl}=a_{ij}b_{jk}c_{k\ell}d_{i\ell} \mbox{ for }  1\leq i,j,k, \ell \leq 2.\]
This toric variety has
dimension $8$ and degree $64$. The multidegree 
of its conormal variety is
$$ 48s^{15}t + 192s^{14}t^2 +576s^{13}t^3 + 1056s^{12}t^4+ 1440s^{11}t^5 +1344s^{10}t^6 + 864s^9t^7 +328 s^8t^8+  64s^7t^9.$$

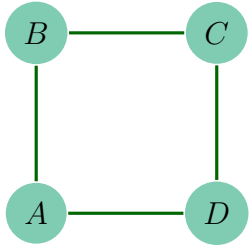
\begin{figure}
         \centering
         \begin{tikzpicture}
  [scale=.3,auto=left,very thick,-]
\node [circle, fill=green!60!blue!50!white] (nA) at (-4,0) {$A$};
\node [circle, fill=green!60!blue!50!white] (nC) at (4,8)  {$C$};
\node [circle, fill=green!60!blue!50!white] (nD) at (4,0)  {$D$};
\node [circle, fill=green!60!blue!50!white](nB) at (-4,8)  {$B$};
\draw (nB) -- (nA) [green!40!black]  node [midway, left, black  ] {};
\draw (nB) -- (nC) [green!40!black] node [midway, right, black ] {};
\draw (nC) -- (nD) [green!40!black] node [midway, right, black ] {};
\draw (nA) -- (nD) [green!40!black]  node [midway, left, black  ] {};
\end{tikzpicture}
    \caption{The smallest nondecomposable graph (the $4-cycle$ graph).}
    \label{fig:4-cycle-graph}
\end{figure}

\subsubsection{Polar degrees of the no-three-way interaction models}
The no-three-way interaction model is a toric 
model based on the graph $G$, which is a triangle, though it is not a graphical model. It belongs to the larger class of hierarchical log-linear models; see \cite{sullivant2018algstat}.
If the three random variables have state spaces $[r]$, $[s]$, and $[t]$, the model is parametrized by
\[
p_{ijk} =a_{ij}b_{jk}c_{ik} \mbox{ where } i\in [r],\, j \in [s], \, k\in [t].
\]
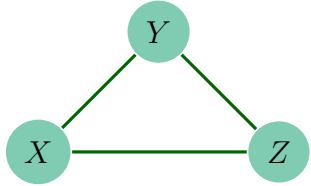
\begin{figure}[ht!]
    \centering
    \begin{tikzpicture}
    [scale=.4,auto=left,very thick,-]
\node [circle, fill=green!60!blue!50!white] (nA) at (-4,0) {$X$};
\node [circle, fill=green!60!blue!50!white] (nC) at (4,0)  {$Z$};
\node [circle, fill=green!60!blue!50!white](nB) at (0,4)  {$Y$};
\draw (nA) -- (nB) [green!40!black]  node [midway, left, black  ] {};
\draw (nB) -- (nC) [green!40!black] node [midway, right, black ] {};
\draw (nA) -- (nC) [green!40!black] node [midway, right, black ] {};
\end{tikzpicture}
    \caption{Graph of the no-three-way interaction model.}
    \label{fig:No3WayGraph}
\end{figure}

We were able to compute the polar degrees when
 $r=s=t=2$ and $r=s=2, \, t=3$. 
 The first variety is a hypersurface in $\mathbb{P}^7$ of degree $4$. Its polar degrees can be read from 
 $$ 4s^6t^2 + 12s^5t^3 + 36s^4t^4 + 36s^3t^5 + 36s^2t^6 + 12st^7 + 4t^8.$$
 The second model is $9$-dimensional 
 in $\mathbb{P}^{11}$ and its degree is $12$.
 The respective 
 multidegree of its conormal variety is  
 $$12s^{10}t^2 + 56s^9t^3 + 180s^8t^4 + 288s^7t^5 + 376s^6t^6 + 288s^5t^7 + 180s^4t^8 + 56s^3t^9 + 12s^2t^{10}.$$
 
 Based on these computations, we venture to make the following conjecture. 
\begin{conjecture}
  Let $X_{(2,2,k)}^3$ be the no-three-way interaction model of two binary random variables and a third one with state space $k$. Then, the sequence of its polar
  degrees is palindromic. 
\end{conjecture}

\section{Computing Wasserstein Degrees} \label{section3}

In this section, our objective is to solve 
\begin{equation*} 
\mbox{minimize} \,\, \ell_F(\nu) \,\,\,
\mbox{subject to} \,\,  \nu \in (\mu + \mathcal{L}_F) \cap \M
\end{equation*}
algebraically. Here $\M = X \cap \Delta_{n-1}$
is a toric model, and $F$ is a face of the 
Wasserstein ball $B$ for a Wasserstein distance 
based on a  metric on $[n]$. The 
objective function $\ell_F$ is a linear functional that attains its maximum on $F$ over $B$. A natural choice is given by a normal vector $n$ to the face $F$ and defining $\ell_F(\nu):= n^\top \nu$. Finally, $\mu + \mathcal{L}_F$ is the affine linear subspace that is a translation of $\mathcal{L}_F$ by 
a generic point $\mu \in \Delta_{n-1}$, where $\mathcal{L}_F$ is the subspace spanned by the vertices of $F$. This is a constrained optimization problem with two types of constraints: linear constraints defining the affine space $\mu + \mathcal{L}_F$ and polynomials defining $\M$; here, we are ignoring the non-negativity constraints on the coordinate variables. Since the feasible 
set is non-convex, we solve this problem by 
computing the complex critical points of $\ell_F$ on $(\mu + \mathcal{L}_F) \cap \M$. Recall that we defined the Wasserstein degree $w(X,F)$ of $X$ with respect to $F$  as the number of these complex critical points for  generic 
$\mu \in \Delta_{n-1}$.

One way of computing these complex critical points is by introducing Lagrange multipliers. This adds extra variables, as many
as there are constraints. Instead, we use an equivalent formulation without the extra Lagrange multipliers, but at the expense of introducing more equations. We now explain this.

We assume that $(\mu + \mathcal{L}_F) \cap \M$ is defined by the equations
$$ g_1(p) \, = \, g_2(p) \, = \, \cdots \, = g_m(p) \, = \, 0,$$
where $g_i \in \mathbb{Q}[p_1, \ldots, p_n]$. 
We note that one of the polynomials is $p_1 + p_2 + \cdots + p_n - 1$. By definition,
$\nu^*$ is a critical point of $\ell_F$ on 
$(\mu + \mathcal{L}_F)  \cap \M$ if 
\begin{itemize}
\item[a)] $(\nabla_p \ell_F) (\nu^*) \in \mathrm{span}_{\C} \left\{ (\nabla_p g_1)(\nu^*), \, (\nabla_p g_2)(\nu^*), \ldots, (\nabla_p g_m)(\nu^*) \right\}$, and 
\item[b)] $g_1(\nu^*) = g_2(\nu^*) = \cdots =  g_m(\nu^*) = 0.$
\end{itemize}
Let $I = \langle g_1, \ldots, g_m\rangle$ be the 
ideal generated by the defining equations and let\\ $c = \codim(I) + 1$. Then, the first condition 
above holds if and only if  the 
augmented Jacobian matrix evaluated at $\nu^*$,
$$ \overline{\mathrm{Jac}}(\nu^*) \, := \, 
\begin{bmatrix}
   (\nabla_p \ell_F) (\nu^*) \\
    (\nabla_p g_1)(\nu^*) \\
    (\nabla_p g_2)(\nu^*) \\
    \vdots \\
    (\nabla_p g_m)(\nu^*) 
\end{bmatrix}$$

has rank at most $ c-1$. This is equivalent to all $c$-minors of $\overline{\mathrm{Jac}}(\nu^*)$ vanishing. In other words, the critical points
are the complex solutions to the system defined
by $g_1 = \ldots = g_m =0$ together with setting 
the $c$-minors of $\overline{\mathrm{Jac}}(p)$ 
equal to zero. We now present this formulation as an algorithm. This algorithm is employed in our computational experiment, as presented in the next subsection.
\begin{algorithm} 
\caption{Wasserstein Degree without Lagrange Multipliers}\label{alg:cap}
\begin{algorithmic}
\State {\bf Input}: Equations of $\M = X  \cap 
\Delta_{n-1}$, a face $F$ of the Wasserstein ball, and a generic point $\mu \in \Delta_{n-1}$.
\State {\bf Output}: Wasserstein degree $w(X,F)$ of $X$ with respect to $F$. \\ 
\State Identify a normal to $F$ that gives $\ell_F$.
\State Let $I=\langle g_1, \ldots, g_m\rangle$ be $I(\M)$ together with the linear equations defining $\mu + \mathcal{L}_F$. 
\State Let $c = n - \dim(I) + 1$.
\State Let $J = I +  \left\langle \text{minors}_{c}\left(\overline{Jac}(p)\right) \right\rangle$. \\
\Return $\mathrm{deg}(J)$.
\end{algorithmic}
\end{algorithm}

\begin{proposition} \label{prop:Wassserstein-degree} The Wasserstein degree $w(X,F)$ does not depend on the choice of $\mu \in \Delta_{n-1}$. 
\end{proposition}
\begin{proof}
The equations defining 
$\mu + \mathcal{L}_F$ are affine linear equations that differ only in their constant terms for different $\mu$. Therefore, $\overline{Jac}(p)$ and the ideal of the $c$-minors of this Jacobian do not 
depend on $\mu$. The codimension of the ideal generated by these minors, together with the equations defining $X$, is equal to the dimension of $\mathcal{L}_F$. The number of critical points is equal to the intersection of the variety defined by the above ideal with $\mu + \mathcal{L}_F$. For generic $\mu$, this number is equal to the degree of this variety, and hence $w(X,F)$ does not depend on the choice of $\mu$. 
\end{proof}
\begin{example}\label{ex3}
We go back to Example \ref{ex:twisted} where $X$ is the twisted cubic in $\Delta_3$ defined by the equations
 \begin{align*}
    g_1 &= 3p_0p_2 - p_1^2 = 0\\
    g_2 &= 3p_1p_3-p_2^2 = 0\\
    g_3 &= 9p_0p_3-p_1p_2 = 0\\
    g_4 &= p_0+p_1+p_2+p_3 - 1 =0.
\end{align*} 
We use the discrete metric on $[4]=\{1,2,3,4\}$ and take
 $ \mu = \left(\frac{1}{6},\frac{1}{2},\frac{1}{6},\frac{1}{6}\right)$. Consider the edge $F$ of the Wasserstein ball determined by the vertices $(0,0,1,-1)$ and $(1,0,0,-1)$. Then $ \mu + \mathcal{L}_F$ is defined by the equations 
 $g_4 = 0$, together with $g_5 = p_0+p_2+p_3-\frac{1}{2} = 0$. We can take the linear functional $\ell_F= p_0+p_2-2p_3$. Therefore, the optimization problem we want to solve is to minimize $\ell_F$ over $X\cap \mathcal{L}_F$, and the corresponding Wasserstein degree to be computed is $w(X,F)$. To solve the problem by the Lagrange multiplier method, one can start by finding a Gr\"obner basis for $X\cap (\mu+\mathcal{L}_F)$ using \texttt{Macaulay2} \cite{M2} as follows:
\begin{lstlisting}
i1 : R = QQ[p_0, p_1, p_2, p_3, MonomialOrder => Lex];
i2 : l = p_0 + p_2 - 2 * p_3;
i3 : g1 = 3*p_0*p_2 - p_1^2;
i4 : g2 = 3*p_1*p_3 - p_2^2;
i5 : g3 = 9*p_0*p_3 - p_1*p_2;
i6 : g4 = p_0 + p_1 + p_2 + p_3 - 1;
i7 : g5 = p_0 + p_2 + p_3 - 1/2;
i8 : I = ideal{g1,g2,g3,g4,g5};
i9 : gens gb I
o9 = | 216p_3^3-540p_3^2+18p_3-1 20p_2+36p_3^2-72p_3-3 2p_1-1
     ---------------------------------------------------------
     20p_0-36p_3^2+92p_3-7 |
             1       4
o9 : Matrix R  <--- R
\end{lstlisting}

The Gr\"obner basis comes out to be a set containing: 
\begin{align*}
f_1 = 216p_3^3-540p_3^2+18p_3-&1,\\ 
f_2 = 20p_2+36p_3^2-72p_3-&3,\\
f_3 = 2p_1-&1,\\
f_4 =  20p_0-36p_3^2+92p_3-&7.
\end{align*}
We see from here that, 
$\M \cap (\mu+\mathcal{L}_F)$, is zero-dimensional. The intersection consists of three points. Since this variety
consists of finitely many points, they will be 
all critical points. So we conclude that $w(X,F)=3$.
\end{example}

\subsection{Computational experiments}
There are two main sources of bottlenecks for computing the Wasserstein degree without Lagrange multipliers: 
\begin{enumerate}
    \item The computation of the hyperplane representation of the Wasserstein ball $B$ from its vertex representation 
    and the enumeration of its faces, and
    \item the large number of $c$-minors
    of $\overline{\mathrm{Jac}}(p)$ in \Cref{alg:cap}.
\end{enumerate}
The complexity of the first problem comes from the fact that even for models
in $\Delta_{3}$, such as the binary path model on three vertices, the number 
of 
faces of $B$ is large. 


The
augmented Jacobian matrix $\overline{\mathrm{Jac}}(p)$ is 
a matrix where almost all
the rows are given by the polynomials $g_1, \ldots, g_m$ defining $\M$. These
can grow quickly, and the number
of $c$-minors where $c = \codim(\M) + 1$
is on the order of $\binom{n}{c} \binom{m}{c}$. The computations need to 
be repeated for each face of $B$. 
As a result, the models we could consider were limited.

For implementation, we used SageMath \cite{sagemath} for both its rich library of symbolic manipulation tools and its numerical computational power. 

In the following tables, we report 
the results of our computations for the
binary path model on $3$ vertices, the binary no-three-way interaction model, and various Hirzebruch surfaces. 
For the first two, the distance metric 
is the Hamming distance on $[2] \times [2] \times [2]$. That is, 
\[d=
    \begin{pNiceMatrix}[first-row,first-col]
        & 000 & 001 & 010 & 011  & 100 & 101 & 110 & 111\\
    000 &  0  &  1  &  1  & 2    &  1  &  2  &  2  & 3  \\
    001 &  1  &  0  &  2  & 1    &  2  &  1  &  3  & 2  \\
    010 &  1  &  2  &  0  & 1    &  2  &  3  &  1  & 2  \\
    011 &  2  &  1  &  1  & 0    &  3  &  2  &  2  & 1  \\
    100 &  1  &  2  &  2  & 3    &  0  &  1  &  1  & 2  \\
    101 &  2  &  1  &  3  & 2    &  1  &  0  &  2  & 1  \\
    110 &  2  &  3  &  1  & 2    &  1  &  2  &  0  & 1  \\
    111 &  3  &  2  &  2  & 1    &  2  &  1  &  1  & 0  \\
\end{pNiceMatrix}\]
is the metric used for the computations reported in \Cref{tab:AlgebraicPathModel222} and \Cref{tab:AlgebraicNoThreeWay}. For the last, it 
is the $L_1$-metric on $[n]$ where
$n= a+b+2$ for the Hirzebruch surface $S(a,b)$.
We report the frequency of the Wasserstein degrees observed for each model and 
each $i$-dimensional face of the corresponding Wasserstein ball $B$. We use the notation $k \, : \, s$ where $s$ is the number of faces of the appropriate dimension or codimension with Wasserstein distance degree $w(X,F) = k$. An entry of the form $- \, : \, s$ indicates that the ideal $J$ in Algorithm 1 is not zero dimensional. Each table contains the $f$-vector $(f_0, f_1, \ldots)$ of $B$ where $f_i$ is the number of $i$-dimensional faces.

The Wasserstein degrees are rarely equal to the polar degrees. This results from the fact that the linear spaces we get from the faces of the Wasserstein unit ball are not required to meet any genericity conditions, unlike the linear spaces used in the definition of polar degrees. 

We would like to point out that, ultimately, the real critical points are most relevant to solve the Wasserstein distance problem. The number of the real critical points will stay constant for points $\mu$ in certain regions of the probability simplex $\Delta_{n-1}$. These regions should be the connected regions of the complement of a discriminant hypersurface, reminiscent of the Euclidean distance degree discriminants \cite{draisma2016euclidean}. We leave this to a future project. 

\section*{Acknowledgements}
Part of this research was performed while the authors were visiting the Institute for Mathematical and Statistical Innovation (IMSI), which is supported by the National Science Foundation (Grant No. DMS-1929348).

We thank Bernd Sturmfels and Jose Israel Rodriguez for their helpful conversations. We also appreciate the valuable comments of the anonymous referee.

We thank the University of Hawaii Information Technology Services -- Cyberinfrastructure for providing the advanced computing resources and technical assistance utilized in computing polar degrees of some varieties that appeared in this work. These resources were partly made possible by National Science Foundation CC* awards \#2201428 and \#232862. IN was supported by the National Science Foundation grant DMS-1945584.

We also thank the support given by the UC Davis TETRAPODS Institute of Data Science. This research was particularly supported by NSF HDR:TRIPODS grant CCF-1934568. 

\begin{table}[h!]
\tiny
\begin{tabular}{l|l}
\hline 
$f$-vector:   & (24,204,812,1674,1836,1008,216) \\ \hline 
dimension 1 & $\begin{cases}
 0 : 12 \\
 1 : 150 \\
 2 : 42 
 \end{cases}$   \\ \hline 
dimension 2 & $\begin{cases}
 0 : 162 \\
 1 : 16 \\
 2 : 416 \\
 4 : 98 \\
 6: 52 \\
 7: 16 \\
 8 : 8 \\
 - : 44
 \end{cases}$  \\ \hline 
dimension 3 & $\begin{cases}
 0 : 100 \\
 1 : 254 \\
 2 : 68 \\
 - : 1252
 \end{cases}$  \\ \hline 
dimension 4 & $\begin{cases}
 0 : 94 \\
 2 : 146 \\
 - : 1596
 \end{cases}$ \\
\hline 
\end{tabular}
\caption{Wasserstein degrees of the binary path model on three vertices under Hamming distance.}
\label{tab:AlgebraicPathModel222}
\end{table}

\begin{table}[h!]
\centering
\tiny
\begin{tabular}{l|l}
\hline 
$f$-vector:   & (24,192,652,1062,848,306,38) \\ \hline 
dimension 0 & $\begin{cases}
 1 : 24
 \end{cases}$ \\ \hline 
dimension 1 & $\begin{cases}
 2: 192
 \end{cases}$   \\ \hline 
dimension 2 & $\begin{cases}
 2: 84 \\
 6 : 568
 \end{cases}$  \\ \hline 
dimension 3 & $\begin{cases}
 9 : 48 \\
 10 : 96 \\
 12 : 288 \\
 14: 48 \\
 16: 30 \\
 19: 288 \\
 -: 264
 \end{cases}$ \\ \hline  
\end{tabular}
\caption{Wasserstein degrees of binary no-three-way interaction model (2,2,2) under Hamming distance.}
\label{tab:AlgebraicNoThreeWay}
\end{table}
\newpage 

\begin{table}[h!]
\centering
\tiny
\begin{tabular}{p{0.05\linewidth} |
p{0.22\linewidth} |p{0.17\linewidth} |p{0.17\linewidth} |p{0.17\linewidth}}
\hline
$a,b$ &  $f$-vector & Codimension 3 
 & Codimension 2 & Codimension 1  \\ \hline
 1,2 & (8, 24, 32, 16) & $\begin{cases}
 1 : 8 \\
 2 : 8 \\
 3 : 8
 \end{cases} $
 & $\begin{cases}
 1 : 4 \\
 2 : 4 \\
 3 : 12 \\
 4 : 12
 \end{cases}$ & 
 $\begin{cases}
     3 : 16
 \end{cases}$ \\ \hline
  1,3 & (10, 40, 80, 80, 32) & $\begin{cases}
 1 : 16 \\
 2 : 24 \\
 3 : 24 \\
 4: 16
 \end{cases} $
 & $\begin{cases}
 0 : 2 \\
 2 : 2 \\
 3 : 16 \\
 4 : 46 \\
 6 : 14
 \end{cases}$ & 
 $\begin{cases}
     3 : 32
 \end{cases}$ \\ \hline
   1,4 & (12, 60, 160, 240, 192, 64) & $\begin{cases}
 1 : 32 \\
 2 : 48 \\
 3 : 32 \\
 4: 64 \\
 5: 64
 \end{cases} $
 & $\begin{cases}
 5 : 80 \\
 6 : 16 \\
 7 : 48 \\
 8 : 48 \\
 \end{cases}$ & 
 $\begin{cases}
     5 : 64
 \end{cases}$ \\ \hline
    1,5 & (14, 84, 280, 560, 672, 448, 128) & $\begin{cases}
 1 : 64 \\
 2 : 96 \\
 3 : 96 \\
 4: 224 \\
 5: 128 \\
 6: 64
 \end{cases}$ 
 & $\begin{cases}
 3 : 12 \\
 5 : 64 \\
 6 : 64 \\
 7 : 52 \\
 8 : 192 \\
 10 : 64 \\
 \end{cases}$ & 
 $\begin{cases}
     5 : 128
 \end{cases}$ \\ \hline
     2,3 & (12, 60, 160, 240, 192, 64) & $\begin{cases}
 1 : 32 \\
 2 : 80 \\
 3 : 32 \\
 4: 32 \\
 5: 64 
 \end{cases} $
 & $\begin{cases}
 4 : 16 \\
 5 : 16 \\
 6 : 48 \\
 7 : 80 \\
 8 : 32 \\
 \end{cases}$ & 
 $\begin{cases}
     5 : 64
 \end{cases}$ \\ \hline
      2,4 & (14, 84, 280, 560, 672, 448, 128) & $\begin{cases}
 1 : 64 \\
 2 : 160 \\
 3 : 64 \\
 4: 128 \\
 5: 64 \\
 6: 192
 \end{cases} $
 & $\begin{cases}
 2 : 4 \\
 4 : 4 \\
 6 : 44 \\
 7 : 128 \\
 8 : 96 \\
 9: 64
 \end{cases}$ & 
 $\begin{cases}
     6 : 128
 \end{cases}$ \\ \hline
\end{tabular}
\caption{Wasserstein degrees of various Hirzebruch surfaces under $L_1$ metric.}
\label{tab:AlgebraicHirzebruch}
\end{table}
{ \text{ }}
\newpage

\bibliographystyle{plain}
\bibliography{refs.bib}
\end{document}